\documentclass[reqno]{amsart}
\usepackage{amsfonts,amssymb}
\usepackage{amsfonts,euscript}
\usepackage{xcolor}
\usepackage{tikz}

\usetikzlibrary{positioning,arrows,automata,shadows,petri}
\newtheorem{remark}{Remark}

\newtheorem{lemma}{Lemma}
\newtheorem{theorem}{Theorem}
\newtheorem{definition}{Definition}

\newtheorem{cor}{Corollary}
\newtheorem{proposition}{Proposition}
\newtheorem{example}{Example}
\numberwithin{equation}{section}

\begin{document}
\title[On the regular sum-free sets]{On the regular sum-free sets$^\star$}
\date{\today}
\author[Z.-X. Wen]{Zhi-xiong Wen}
\author[J.-M. Zhang]{Jie-meng Zhang}
\author[W. Wu]{Wen Wu$^\dagger$}
\thanks{$\star$ This work is supported by NSFC (Grant Nos. 11401188,11371156).\\
\indent $\dagger$ Wen Wu is the corresponding author.
\emph{E-mail}: wuwen@hubu.edu.cn. }

\address[Z.-X. Wen and J.-M. Zhang]{School of Mathematics and Statistics, Huazhong University of Science and Technology, Wuhan 430074, P.R. China}
\email{zhi-xiong.wen@hust.edu.cn \textrm{(Z.-X. Wen)}, zhangfanqie@hust.edu.cn \textrm{(J.-M. Zhang)}}
\address[W. Wu]{Department of Mathematics, Hubei University, Wuhan 430062, P.R. China}
\email[Corresponding author]{wuwen@hubu.edu.cn; hust.wuwen@gmail.com \textrm{(W. Wu)}}

\begin{abstract}
Cameron introduced a bijection between the set of sum-free sets and
the set of all zero-one sequences. In this paper, we study the sum-free sets
of natural numbers
corresponding to certain zero-one sequences which contain the Cantor-like
sequences and some substitution sequences, etc. Those sum-free sets
considered as integer sequences are $2$-regular. We also prove that
sequences corresponding to certain sum-free sets are automatic.
\end{abstract}

\maketitle

\section{Introduction}
A set of integers, denoted by $S$, is called a \emph{sum-free set} if $%
S\cap(S+S)=\emptyset$, where $S+S$ denotes the set of pairwise sums, i.e., $%
S+S=\{x+y:x,y\in S\}$. Thus, for any sum-free set $S$, there do not exist $%
x,y,z\in S$ for which $x+y=z$. It is natural to arrange the elements of a
sum-free set $S$ in ascending order, so $S=(S_{n})_{n\geq 0} $ can also
be treated as an integer sequence.

Cameron and Erd\H{o}s \cite{Ca,ca} studied the number of sum-free sets which
are contained in the first $n$ integers. They showed that the number of
sum-free subsets of $\{\frac{n}{3},\frac{n}{3}+1,\cdots,n\}$ is $%
O(2^{\frac{n}{2}})$. Calkin \cite{cs} showed that
the Hausdorff dimension of the sum-free sets is at most $0.599$. 
Calkin \cite{number} proved that the number of the sum-free subsets of
 $\{1,2,\cdots,n\}$ is $\emph{o}(2^{n(1/2+\epsilon)})$ for every $%
\epsilon>0$. {\L}uczak and Schoen \cite{TT} studied the properties of
$k$-sum-free sets with precise upper density.

There are several methods to construct infinite sum-free sets. One of them
is to construct such a set directly using its definition; that is, to construct it from numbers which are not the sum of two earlier numbers. Another way is to
construct the sum-free set from an infinite zero-one sequence. This method
was introduced by Cameron. Cameron \cite{Pj} defined a
natural bijection between $\Sigma$ and $\mathfrak{S}$, where $\Sigma$ and $%
\mathfrak{S}$ denote the set of all zero-one sequences and the set of all
sum-free sets respectively.

\subsection{The bijection between $\Sigma$ and $\mathfrak{S}$.}

Let $S\in\mathfrak{S}$ be a sum-free set and $\mathbf{v}$ be a ternary
sequence defined as follow:
\begin{equation}  \label{vn}
v_n=\left\{
\begin{array}{ll}
1, & \hbox{if $n\in S$;} \\
\ast, & \hbox{if $n\in S+S$;} \\
0, & \hbox{otherwise.}%
\end{array}
\right.
\end{equation}
By deleting all the $\ast$'s in $\mathbf{v}$, we obtain a unique zero-one
sequence $\mathbf{v}^{\prime}$. This process introduces a mapping from $%
\mathfrak{S}$ to $\Sigma$. Calkin \cite{Nc} showed that this mapping is a
bijection and denoted its inverse by $\theta$. Thus $\theta:\mathbf{v}%
^{\prime}\mapsto S$ is a bijection from $\Sigma$ to $\mathfrak{S}$.

Since there is a one-to-one correspondence between the set of zero-one sequences and the set of sum-free sets (of positive integers), one would like to know the connection between these two.

\subsection{Periodicity of $S$}

A sum-free set $S$ is said to be \emph{ultimately periodic} if there exist
positive integers $m,n_{0}$ such that for all $n>n_{0}$, $n\in S$ if and
only if $n+m\in S$. If $n_{0}=0$, then $S$ is \emph{periodic}. It is not
easy to show whether a given sum-free set is periodic or not.

In \cite{Pc}, Cameron observed that if a sum-free set is (ultimately)
periodic, the corresponding zero-one sequence is also (ultimately) periodic.
This was proved by Calkin and Finch in \cite{Nc}. 
Conversely, Cameron \cite{Nc} also asked whether sum-free sets
corresponding to (ultimately) periodic zero-one sequences are (ultimately)
periodic or not. This question is still open. With the help of a computer,
Calkin and Finch \cite{Nc} presented some sum-free sets, which correspond to
periodic zero-one sequences, and appear to be aperiodic (aperiodicity checked up
to $10^7$). There is no proof to show whether such sum-free sets are periodic or aperiodic.
Calkin and Erd\H{o}s \cite{aperiodic} showed that a class of aperiodic
sum-free sets $S$ is incomplete, i.e., $\mathbb{N}\backslash (S+S)$ is an
infinite set. Later, Calkin, Finch and Flowers \cite{difference} introduced
the concept of difference density which can be used to test whether specific sets
are periodic or not. These tests produced further evidence that certain sets are not ultimately periodic. Payne \cite{GP}
studied the properties of some sum-free sets over an additive group.

\subsection{Main results.}

Motivated by Cameron's question, in this paper, we will investigate the
regularity of $S$. We would like to know what kind of zero-one sequences
will lead to $k$-regular sum-free sets?

\begin{definition}
A sum-free set $S$ is called a \emph{$k$-regular sum-free set} if the
sequence $(S_n)_{n\geq 0}$ is $k$-regular\footnote{%
For the definition of $k$-regular sequences, refer to Definition \ref{def:regular} in
Section \ref{reg}.} for some $k\in\mathbb{N}$.
\end{definition}

Let $\mathbf{c}=c_{0}c_{1}c_{2}\cdots $ be a zero-one sequence with $c_0=1$.
Let $\mu_{n}$ denote the number of zeros between the $n$-th and the $(n+1)$-th ones in $%
\mathbf{c}$. An equivalent definition of $\mu_n$ is
in Definition \ref{deff}. The sum-free set corresponding to $%
\mathbf{c}$ is denoted by $S$. Our results are stated as follows:
\bigskip

\noindent \textbf{Main results I.} \emph{Suppose the integer sequence
$(\mu _{n})_{n\geq 1}$ is $%
2 $-regular and satisfies the conditions
\begin{equation}
\left\{
\begin{array}{l}
\mu _{2^{m}}>\sum_{i=1}^{2^{m}-1}\mu _{i}+2^{m}+\frac{3^{m}-1}{2}, \\
\mu _{2^{m}+k}=\mu _{k},\quad \forall ~0<k<2^{m},%
\end{array}%
\right.  \label{hm}
\end{equation}%
} \emph{for all $m\geq 1$, then $S$ is a $2$-regular sum-free set.}

\bigskip

For any fixed $l_{1}\geq 0,~l_{2}\geq 0$ and $l_{3}\geq 3$, let $\sigma
(l_{1},l_{2},l_{3})$ be a substitution over the alphabet $\{0,1\}$ given by
\begin{equation}
1\mapsto 1\overbrace{0\cdots 0}^{l_{1}}1\overbrace{0\cdots 0}^{l_{2}},\quad
\ \ \ 0\mapsto \overbrace{0\cdots 0}^{l_{3}}.  \label{sigma}
\end{equation}%
The sequence
\begin{eqnarray}
\mathbf{c}_{l_{1},l_{2},l_{3}} &=&\sigma (l_{1},l_{2},l_{3})^{\infty }(1)
\nonumber \\
&=&10\cdots 010\cdots 010\cdots 01\cdots  \label{cm}
\end{eqnarray}%
is called a \emph{Cantor-like sequence.} In particular, $\mathbf{c}_{1,0,3}$ is the Cantor sequence. In the rest of this paper, since $%
l_{1},l_{2},l_{3}$ are fixed in the context, we will simply use $\mathbf{%
c}$ to denote such Cantor-like sequences. It is worth pointing out that
Cantor-like sequences contain a class of automatic sequences and a class of
non-automatic substitution sequences.

\medskip

\noindent \textbf{Main results II.} \emph{Let }$\mathbf{c}$\emph{\ be a
Cantor-like sequence satisfying the assumption in Main results $I$.
Then the corresponding sum-free set $S$ modulo $2$ is
the Thue-Morse sequence\footnote{The Thue-Morse sequence beginning by $1$
is the infinite sequence $\mathbf{t}=t_0t_1\cdots\in\{0,1\}^{\mathbb{N}}$
satisfying the following relations: $t_0=1,t_{2n}=t_n,t_{2n+1}=1-t_n$.}
beginning by $1$
up to a coding.}

\medskip

We also investigate the zero-one sequences corresponding to a class of
sum-free sets generated from the set $(bn+1)_{n\geq 0}$ through base
changing where $b\in\mathbb{N}$ and $b\geq 2$. For example, let $S=\{1,3,5,7,\cdots \}$ where $%
S_{n}=(2n+1)_{n\geq 0}$. Suppose $n=\sum_{i=1}^{k}n_{i}\cdot 2^{i-1}$, with$%
~n_{i}\in \{0,1\}$ for any $1\leq i\leq k$, denote $S_{n}^{\prime
}:=\sum_{i=1}^{k}n_{i}\cdot 3^{i}+1$. Then $S^{\prime
}=\{1,4,10,13,28,31,37,40,\cdots \}$ is also sum-free. Moreover, the
corresponding zero-one sequence of $S^{\prime }$ is the Cantor sequence (see Remark \ref{rem:cantor}).

\medskip

\noindent\textbf{Main results III.} \emph{Suppose $b\geq2$, let $S$ be the
sum-free set given by $S_n=\sum_{i=1}^kn_i(2b-1)^i+1$ where $%
n=\sum_{i=1}^kn_i2^{i-1}$ with $n_i\in\{0,1\}$ , then $
\theta^{-1}(S)$ is an automatic sequence.}

\medskip

This paper is organized as follows. In Section 2, we introduce some
definitions and auxiliary
lemmas which are useful
in the proof of Theorem \ref{mainthm}. In Section 3, we give the explicit value of the sum-free
set $S$ corresponding to $\mathbf{c}$ and discuss the regularity of $S$.
A class of sum-free sets $S$ corresponding to Cantor-like sequences
is investigated in Section 4. In the last section, we discuss the automaticity of zero-one
sequences corresponding to certain sum-free sets.

\section{Preliminary}

\if
It is of interest to know whether the corresponding sum-free sets of
automatic binary sequence is automatic, this work is intended as an attempt
to motivate that is true. Cobham \cite{Ch} presented an important connection
between automatic sequences and fixed points of substitution of constant
length. We introduce the notation of substitution sequence. \fi

\subsection{$m$-complement of nonnegative integers.}

For $b\in \mathbb{N}$, let $\Sigma_b :=\{0,1,\cdots,b-1\}$. Let $\Sigma_b^k$ be
the set of words of length $k$ over $\Sigma_b$ and $\Sigma_b^{\ast}=\bigcup_{k%
\geq0}\Sigma_b^k$. The \emph{$b$-ary expansion} of $n$ is denoted by
\[
(n)_b:=\epsilon_k\epsilon_{k-1}\cdots\epsilon_1\in\Sigma_b^{k},
\]
where $\epsilon_k\neq 0$. The \emph{$b$-ary number} of $\mathbf{w}%
=w_kw_{k-1}\cdots w_1\in\Sigma_b^k$ is given by
\[
[\mathbf{w}]_b:=\sum_{i=1}^{k}w_i\cdot b^{i-1}.
\]

For $\mathbf{w}=w_kw_{k-1}\cdots w_1\in\Sigma_2^k$, the \emph{complement} of $%
\mathbf{w}$ is the word
\[
\bar{\mathbf{w}}=\bar{w}_k\bar{w}_{k-1}\cdots \bar{w}_1,
\]
where $\bar{\epsilon}=1-\epsilon$ for any $\epsilon\in\Sigma_2$.

Let $m\in\mathbb{N}$ be fixed and $n=\epsilon_m2^{m-1}+%
\cdots+\epsilon_12^0<2^m$ with $\epsilon_i\in\Sigma_2$. In this case, $%
\epsilon_m\epsilon_{m-1}\cdots\epsilon_1$ is a $2$-ary representation of $n$
of length $m$ which allows leading zeros, denoted by $(n)_{2,m}$.
The \emph{$m$-complement} of $n$,
denoted by $\bar{n}$, is given by
\[
\bar{n}=\left[\overline{(n)}_{2,m}\right]_2=\bar{\epsilon}_m2^{m-1}+\cdots+\bar{\epsilon}_12^0.
\]
For example, when $m=3$, then $\bar{1}=1\cdot 2^2+1\cdot 2^1+0\cdot 2^0=6$, $%
\bar{3}=2^2=4$. We also use notation `$0n$' and `$1n$' in form, which
denote integers of $2$-ary representation `$0\epsilon_m\epsilon_{m-1}\cdots%
\epsilon_1$' and `$1\epsilon_m\epsilon_{m-1}\cdots\epsilon_1$' respectively.
In fact,
\[
0n=n,\quad 1n=2^m+n.
\]

\subsection{Two auxiliary bijections}

Let $(h(i))_{i\geq 1}$ be a strictly increasing positive integer sequence. Define
the mapping $f: \Sigma_{2}^m\times\Sigma_{2}^m\rightarrow\mathbb{N}$ by
\[
(x,y)\mapsto f(x,y)=\sum_{i=1}^m(x_i+y_i)h(i)+2.
\]
Let $\sim$ be the equivalence relation on $\Sigma_{2}^m\times\Sigma_{2}^m$
induced by $f$, i.e.,
\[
(x,y)\sim(x^{\prime},y^{\prime}) \textrm{ if and only if } f(x,y)=f(x^{\prime},y^{%
\prime}).
\]
Denote $\widetilde{\Sigma}^m:=(\Sigma_2^m\times\Sigma_2^m)/{\sim}$ for
simplicity. Hence, we have defined an injection
\[
f:\widetilde{\Sigma}^m\rightarrow\mathbb{N}.
\]

For any $m\geq 0$, define two mappings $\varphi_m,~\psi_m$ as follows:
for any $(u,v)\in%
\widetilde{\Sigma}^m$,
\begin{eqnarray*}
(u,v) &\mapsto& \varphi_m(u,v)=(\bar{u},\bar{v}), \\
(u,v) &\mapsto& \psi_m(u,v)=(0u,1v).
\end{eqnarray*}
Notice that
\begin{eqnarray}
f\circ\varphi_m(u,v)&=&f(\bar{u},\bar{v})  \nonumber \\
&=&\sum_{i=1}^m (\bar{u}_i+\bar{v}_i)h(i)+2  \nonumber \\
&=&2\sum_{i=1}^mh(i)-f(u,v)+4,  \label{eqn:phi}
\end{eqnarray}
and
\begin{equation}
f\circ\psi_m(u,v)=f(0u,1v)=f(u,v)+h(m+1).  \label{eqn:psi}
\end{equation}
By (\ref{eqn:phi}) and (\ref{eqn:psi}),
\[
(\bar{u},\bar{v})\sim(\bar{u}^{\prime},\bar{v}^{\prime})%
\Longleftrightarrow(u,v)\sim (u^{\prime},v^{\prime}) \Longleftrightarrow
(0u,1v)\sim (0u^{\prime},1v^{\prime}).
\]
Thus $\varphi_m$ and $\psi_m$ are well-defined injections.

But these two mappings are not always bijections. Additional conditions are
needed to ensure that. Before this, we need the following notation: for any
$0\leq k<2^m$ with $k=\epsilon_m2^{m-1}+\cdots+\epsilon_12^0$ and $%
\epsilon_i\in\Sigma_2$, denote
\begin{equation}
M_k:=\sum_{i=1}^m\epsilon_ih(i)+1, ~~M_{2^m+k}:=M_k+h(m+1).
\label{lem1eqn31}
\end{equation}
Note that $M_{2^m-1}=\sum_{i=1}^mh(i)+1$. Moreover, for any $m\geq0, 0\leq k<2^m$%
, let
\begin{eqnarray*}
L_m &:=& \{(u,v)\in\widetilde{\Sigma}^m:~f(u,v)\leq M_{2^m-1}\}, \\
R_m &:=& \{(u,v)\in\widetilde{\Sigma}^m:M_{2^m-1}+1<f(u,v)\leq 2M_{2^m-1}\},
\\
L(k)&:=& \{(u,v)\in\widetilde{\Sigma}^m:~M_k<f(u,v)<M_k+K\}, \\
R(k)&:=& \{(0u,1v)\in\widetilde{\Sigma}%
^{m+1}:~M_{2^m+k}<f(0u,1v)<M_{2^m+k}+K\},
\end{eqnarray*}
where $K=\mu_k+\alpha_k+1$, and $(\mu_i)_{i\geq 0},~(\alpha_i)_{i\geq 0}$
are positive integer sequences.

\begin{lemma}
\label{lm4} $\varphi_m|_{L_m}$ is a bijection from $L_m$ to $R_m$.
\end{lemma}

\begin{proof}
By (\ref{eqn:phi}), $f(\bar{u},\bar{v})=2M_{2^m-1}-f(u,v)+2.$ Hence, if $%
(u,v)\in L_m$, then
\[
M_{2^m-1}+2\leq f(\bar{u},\bar{v})\leq 2M_{2^m-1},
\]
which implies $\varphi_m(u,v)=(\bar{u},\bar{v})\in R_m$.

Note that $\varphi_m(\bar{u},\bar{v})=(u,v)$. To show $\varphi_m$ is a
surjection, we only need to show $(\bar{u},\bar{v})\in L_m$ for any $%
(u,v)\in R_m$. In fact, if $(u,v)\in R_m$, then
\[
f(\bar{u},\bar{v})=2M_{2^m-1}-f(u,v)+2<M_{2^m-1}+1.
\]
\end{proof}

\begin{lemma}
\label{lm5} For any $0\leq k<2^m$, $\psi_m|_{L(k)}$ is a bijection from $%
L(k) $ to $R(k)$.
\end{lemma}

\begin{proof}
Note that $M_{2^m+k}=h(m+1)+M_k$. By (\ref{eqn:psi}), $(u,v)\in L(k)$ if and
only if
\[
M_{2^m+k}<f(0u,1v)=f(u,v)+h(m+1)<M_{2^m+k}+K.
\]
Thus $\psi_m|_{L(k)}$ is a bijection.
\end{proof}

\begin{remark}
By Lemma \ref{lm4} and Lemma \ref{lm5}, for $0\leq k<2^m$,
\[
\textrm{Card}~L_m=\textrm{Card}~R_m,~\textrm{Card}~L(k)=\textrm{Card}~R(k).
\]
\end{remark}


\section{Regularity of $S$}
Let $\mathbf{w}$ be a zero-one sequence beginning by $1$, and $S=\theta (\mathbf{w})$ be the
corresponding sum-free set. In this section, we will characterize $S$ using the properties of $S+S$.

\subsection{Elementary observations of $S$}

Now assume that $\mathbf{v}=(v_{n})_{n\geq 0}$ is the sequence generated from $%
S $ by (\ref{vn}). In fact, $\mathbf{v}$ labels all the natural numbers by $0
,1$ or $*$. And only those integers belonging to $S$ are labeled by $1$. When
we just care about the number of integers labeled by $0$ or $*$, we will use
``the number of $0$'s (or $*$'s)'' rather than ``the number of integers
labeled by $0$'s (or $*$'s)''. This abuse of language does not cause any
ambiguity in the paper.

\begin{definition}\label{deff}
For any $n\geq 1$, denote by $\mu_n$ (resp. $\alpha_n$), the number of $0$'s
(resp. $*$'s) between $S_{n-1}$ and $S_{n}$. In detail,
\begin{eqnarray*}
&\mu_n := \textrm{Card}~\{i\in\mathbb{N}: v_i=0, S_{n-1}<i<S_{n}\}; \\
&\alpha_n := \textrm{Card}~\{i\in\mathbb{N}: v_i=*, S_{n-1}<i<S_{n}\}.
\end{eqnarray*}
\end{definition}

\begin{remark}
$(1)$ For any $n\geq 1$,
\[
S_{n}-S_{n-1}=\mu_n+\alpha_n+1.
\]
Thus to study $S$, we need to study the properties of $(\mu_n)_{n \geq 1}$ and $%
(\alpha_n)_{n\geq 1}$.

$(2)$ Since $\mathbf{v}$ converts to $\mathbf{w}$ by deleting all the $*$'s, $%
\mu_n$ represents the number of $0$'s between the $n$-th and $(n+1)$-th `$1$%
' in $\mathbf{w}$, i.e.,
\[
\mathbf{w}=1\overbrace{0\cdots 0}^{\mu_{1}}1\overbrace{0\cdots 0} ^{\mu_{2}}1%
\overbrace{0\cdots 0}^{\mu_{3}}1\cdots.
\]
\end{remark}

\label{reg} During the proof of Theorem \ref{mainthm}, we need the following
two notation: for any $n\geq 1$,
\begin{eqnarray}
g(n)&:=&\mu_n+\alpha_n,  \label{gn} \\
h(n)&:=&\left(\sum_{i=1}^{2^{n-1}}g(i)\right)+2^{n-1}.  \label{hn}
\end{eqnarray}
Thus $S_{n}=S_{n-1}+g(n)+1$ and $h(n)$ is strictly increasing.

\begin{lemma}
\label{lem:hn} Suppose $g(2^k+i)=g(i)$ for $0\leq k<m, ~0<i<2^k$. Then
\[
h(m+1)=\sum_{i=1}^{m}h(i)+g(2^m)+1.
\]
\end{lemma}

\begin{proof}
By (\ref{hn}),
\begin{eqnarray*}
&&h(k+1)-h(k) \\
&=&\left(\sum_{i=1}^{2^k}g(i)+2^k\right)-h(k) \\
&=& \left(\sum_{i=1}^{2^{k-1}-1}\big(g(i)+g(2^{k-1}+i)\big)%
+g(2^{k-1})+g(2^k)+2^k\right)-h(k) \\
&=& h(k)+g(2^k)-g(2^{k-1}).
\end{eqnarray*}
Add the last equation up from $k=1$ to $m$.  The result follows.
\end{proof}


\begin{theorem}
\label{thm1} \label{mainthm} Let $\mathbf{c}:=1\overbrace{0\cdots0}^{\mu_1}1%
\overbrace{0\cdots0}^{\mu_2}1\cdots$. Denote its corresponding sum-free set by $%
S=(S_n)_{n\geq 0}$. If the sequence $(\mu_n)_{n\geq1}$ satisfies
 $(\ref{hm})$, then for every integer $n\geq 1$, we have

\begin{enumerate}
\item if $n=2^k(2j+1)$ for some $k,j\geq 0$, then
\begin{equation}  \label{alphan}
\alpha_n=\frac{3^{k}+1}{2};
\end{equation}

\item if $(n)_2=\epsilon_m\epsilon_{m-1}\cdots\epsilon_1$, then
\[
S_n=1+\sum_{i=1}^m\epsilon_ih(i).
\]
\end{enumerate}
\end{theorem}

\begin{proof}[Proof of Theorem \ref{thm1}]
We prove this theorem by induction on $n$. \newline
\textbf{Step1:} Since $\alpha_1=1,~S_1=\mu_1+3$, the conclusion is true for $%
n=1$. \newline
\textbf{Step2:} Assume the result is true for $n<2^m$; that is, if $%
n=2^k(2j+1)$, then $\alpha_n=\frac{3^k+1}{2}$. Thus $\alpha_{2^p+q}=%
\alpha_{q}, \forall~0<q<2^p, 0<p<m$. Moreover, if $(n)_2=\epsilon_m%
\epsilon_{m-1}\cdots\epsilon_1$, then $S_n=1+\sum_{i=1}^m\epsilon_ih(i)$.

Hence, it suffices to show that the result is also true for ~$2^m\leq
n<2^{m+1}$.

\textbf{Step2.1} Let $n=2^m$. By the induction hypothesis and (\ref{alphan}), we have
\begin{eqnarray}
\sum_{i=1}^{2^m-1}\alpha_i
&=& \sum_{i=1}^{2^{m-1}-1}\alpha_i+\alpha_{2^{m-1}}+\sum_{i=1}^{2^{m-1}-1}\alpha_{2^{m-1}+i}
=2\sum_{i=1}^{2^{m-1}-1}\alpha_i+\alpha_{2^{m-1}}\nonumber\\
&=&2^{2}\sum_{i=1}^{2^{m-2}-1}\alpha_i+2\alpha_{2^{m-2}}+\alpha_{2^{m-1}}
\cdots
=\sum_{i=1}^{m}2^{i-1}\alpha_{2^{m-i}}\nonumber\\
&=&\frac{3^{m}-1}{2}.\label{sum:alpha}
\end{eqnarray}

Now we will evaluate $\alpha_{2^m}$. By (\ref{sum:alpha}) and (\ref%
{hm}),
\begin{eqnarray*}
S_{2^m-1} &=& S_0+\sum_{i=1}^{2^m-1}(\mu_i+\alpha_i)+2^m-1 \\
&=& \sum_{i=1}^{2^m-1}\mu_i+2^m+\frac{3^m-1}{2}< \mu_{2^m}.
\end{eqnarray*}
Thus for $i,j\leq 2^m-1$,
\begin{equation}  \label{ss}
S_i+S_j\leq 2S_{2^m-1}<S_{2^m-1}+\mu_{2^m}.
\end{equation}
Thus to evaluate $\alpha_{2^m}$, we only need to count the number of $\ast$'s
between $S_{2^m-1}$ and $2S_{2^m-1}$. It is easy to see that $%
S_0+S_{2^m-1}=S_{2^m-1}+1$ is one of such $\ast$'s. Denote
\begin{eqnarray*}
\widetilde{L}_m &:=& \{x\in S+S:S_0<x<S_{2^m-1}\}, \\
\widetilde{R}_m &:=& \{x\in S+S:S_{2^m-1}+1<x<2S_{2^m-1}\}.
\end{eqnarray*}
Then
\[
\alpha_{2^m}=\textrm{Card}~\widetilde{R}_m+1.
\]
Since $f$ is an injection and $\widetilde{L}_{m}=f(L_m), ~\widetilde{R}%
_{m}=f(R_m)$, by Lemma \ref{lm4}, we have
\[
\textrm{Card}~\widetilde{L}_{m}=\textrm{Card}~\widetilde{R}_{m}.
\]
Hence,
\begin{eqnarray*}
\alpha_{2^m}&=&\textrm{Card}~\widetilde{L}_m+1 \\
&=&\sum_{i=1}^m2^{m-i}\alpha_{2^{i-1}}+1 \\
&=&\frac{3^m+1}{2}.
\end{eqnarray*}

Now we will give the expression of $S_{2^{m}}$. By (\ref{hm}) and (\ref%
{alphan}), for $0<j<m$ and $0<k<2^{j}$,
\[
g(2^{j}+k)=\alpha _{2^{j}+k}+\mu _{2^{j}+k}=\alpha _{k}+\mu _{k}=g(k).
\]%
Let $(2^{m})_{2}=1\overbrace{0\cdots 0}^{m}=:\epsilon _{m+1}(2^{m})\epsilon
_{m}(2^{m})\cdots \epsilon _{1}(2^{m})$. By Lemma \ref{lem:hn} and the
induction assumption, we have
\begin{eqnarray*}
S_{2^{m}} &=&S_{2^{m}-1}+g(2^{m})+1 \\
&=&\sum_{k=1}^{m}h(k)+g(2^{m})+2 \\
&=&h(m+1)+1 \\
&=&1+\sum_{k=1}^{m+1}\epsilon _{k}(2^{m})h(k).
\end{eqnarray*}%
\textbf{Step2.2:} Now we prove the results for $2^{m}<n<2^{m+1}$. Let $%
n=2^{m}+k$ where $1\leq k<2^{m}$, and we prove the results by induction on $k$.
For $1\leq k<2^{m}$, assume the result is true for $n<2^{m}+k$. We will
prove the results for $n=2^{m}+k$.

Firstly, we need to determine the value of $S_{2^{m}+k}$. Denote $\Xi%
_{2^{m}+k-1}:=\{x\in S:x\leq S_{2^{m}+k-1}\}$. Then
\[
\Xi_{2^{m}+k-1}+\Xi_{2^{m}+k-1}=\mathcal{S}_{1}\cup
\mathcal{S}_{2}\cup \mathcal{S}_{3},
\]%
where
\begin{eqnarray*}
\mathcal{S}_{1}:= &\{S_{i}+S_{j}:&i,j<2^{m}\}, \\
\mathcal{S}_{2}:= &\{S_{i}+S_{j}:&i<2^{m},2^m\leq j\leq 2^{m}+k-1\}, \\
\mathcal{S}_{3}:= &\{S_{i}+S_{j}:&2^m\leq i,j\leq 2^{m}+k-1\}.
\end{eqnarray*}

In fact, if $i,j<2^m$, then by (\ref{ss}) we have $S_{i}+S_{j}<S_{2^{m}}$.
Thus $\max \mathcal{S}%
_{1}<S_{2^{m}}$. Note that by the induction hypothesis, $S_{2^{m}+k-1}=h(m+1)+S_{k-1}$ and
$S_{2^{m}}=h(m+1)+1$. Since $(S_{n})_{n\geq 1}$ is strictly increasing, we have for $i,j\geq 2^{m}$,
\begin{eqnarray*}
S_{i}+S_{j} &> &S_{k}+S_{2^{m}}-1 \\
&=&(S_{k-1}+\mu _{k}+\alpha _{k}+1)+(h(m+1)+1)-1 \\
&=&S_{2^{m}+k-1}+\mu _{k}+\alpha _{k}+1=:s.
\end{eqnarray*}%
Therefore $\min \mathcal{S}_{3}> s$. Note that $s=h(m+1)+S_k$. Set $K=g(k)+1$,
then $S_k=S_{k-1}+K$
and $s=S_{2^m+k-1}+K$. Let
\begin{eqnarray*}
& &\widetilde{L}(k):=\{x\in S+S:S_{k-1}<x<S_{k}\}, \\
& &\widetilde{R}(k):=\{x\in S+S:S_{2^{m}+k-1}<x<s\}.
\end{eqnarray*}%
Clearly $f(L(k))=\widetilde{L}(k)$ and $f(R(k))\subset \widetilde{R}(k)$.
Now we will show that $f(R(k))=\widetilde{R}(k)$. Since $f$ is an injection,
by Lemma \ref{lm5}, we have
\[
\textrm{Card}~\widetilde{R}(k)\geq \textrm{Card}~\widetilde{L}(k)=\alpha _{k}.
\]

Suppose $S_{2^m+k}<s$. The way that we construct $S$ from a zero-one sequence implies that there are $\mu_{k}$ $0$'s, at least $\alpha_{k}$ $\ast$'s and at least one $1$ between $S_{2^m+k-1}$ and $s$. However there are only $g(k)$ integers between $S_{2^m+k-1}$ and $s$, which is a contradiction. Therefore $S_{2^m+k}\geq s$. In this case, $f(R(k))= \widetilde{R}(k)$. Hence
$$\textrm{Card}~\widetilde{R}(k)=\alpha _{k}.$$
And the number $s$ must be labeled by $\ast$ or $1$, i.e., $v_s=\ast \textrm{ or } 1$. Since $\max\mathcal{S}_1<S_{2^m}$ and $\min\mathcal{S}_3>s$, then $s\in\mathcal{S}_2$ if $v_s=\ast$. That is,
 $$\exists~ i<2^m,j\geq2^m, \textrm{ s.t., } s=S_i+S_j.$$
Then there exist $i^{\prime},j^{\prime}<2^m$ satisfying $i=0i^{\prime},j=1j^{\prime}$, and
\begin{eqnarray*}
  s=S_k+h(m+1)=S_i+S_j=S_{i^{\prime}}+S_{j^{\prime}}+h(m+1).
\end{eqnarray*}
Thus $S_k=S_{i^{\prime}}+S_{j^{\prime}}$ which contradicts the construction of $S_k$.

Therefore $v_s=1$ and $s=S_{2^m+k}$. Combining this fact with (\ref{hm}), $\alpha_k=\alpha_{2^m+k}$. The proof is completed.
\end{proof}

\begin{remark}
$(h(n))_{n\geq 1}$ is a numeration system (for detail, refer to \cite[Chapter 3]{jps}. In fact, by Lemma \ref{lem:hn},
for any $j\geq1$, $\sum_{i=1}^{j-1}h(i)<h(j)$.
\end{remark}

\begin{definition}[Allouche and Shallit \protect\cite{Jp}]
\label{def:regular} A sequence $(t(n))_{n\geq 0}$ is called \emph{$k$-regular} if
there exist $m$ subsequences of $n$, say $\{n_l^{(j)}\}_{l\geq0}$ $(0\leq j\leq
m-1)$, which satisfy for any $i\geq0$ and $0\leq b<k^i$, the subsequence $%
(t(k^in+b))_{n\geq 0}$ is a $\mathbb{Z}$-linear combination of $t(n_l^{(j)})$.
\end{definition}

\begin{lemma}[Allouche and Shallit \protect\cite{Jp}]
\label{lemmm} If $(u_{n})_{n\geq 1}$ and $(v_{n})_{n\geq 1}$ are both $k$%
-regular, then $(u_{n}+v_{n})_{n\geq 1}$ and $(\sum_{i=1}^{n}u_{i})_{n\geq
1} $ are also $k$-regular.
\end{lemma}

\begin{theorem}
\label{thm:regular}If $(\mu _{n})_{n\geq 1}$ is $2$-regular and
(\ref{hm}) holds, then the sequence $(S_{n})_{n\geq 0}$ is $2$%
-regular.
\end{theorem}

\begin{proof}
Since $(\mu _{n})_{n\geq 1}$ is $2$-regular, by Lemma \ref{lemmm}, so is $%
(\sum_{i=1}^{n}\mu _{i})_{n\geq 1}$. And (\ref{alphan}) in Theorem %
\ref{thm1} implies that
\[
\left\{
\begin{array}{l}
\alpha _{2n}=3\alpha _{n}-1, \\
\alpha _{2n+1}=1.%
\end{array}%
\right.
\]%
Thus $(\alpha _{n})_{n\geq 1}$ is $2$-regular. By Lemma \ref{lemmm}, $(\sum_{i=1}^{n}\alpha _{i})_{n\geq 1}$ is $2$-regular. Note that $(n+1)_{n\geq 1}$ is
also $2$-regular, so by Lemma \ref{lemmm},
\[
S_{n}=\sum_{i=1}^{n}\mu _{i}+\sum_{i=1}^{n}\alpha _{i}+(n+1)
\]%
is $2$-regular.
\end{proof}

\section{Examples}

Let $\Sigma_2^{\ast }$ be the set of finite words over alphabet $\{0,1\}$.
For any $w\in \Sigma_2^{\ast }$, the length of $w$ is denoted by $|w|$. Denote by $|w|_{0}$ and $|w|_{1}$ the number of $0$'s and $1$'s in $w$ respectively.

The following lemma completely characterizes the gap between two adjacent $1$%
's in the Cantor-like sequence $\mathbf{c}$.

\begin{lemma}
\label{lemma1} For any $l_{1}\geq 0,~l_{2}\geq 0,~l_{3}\geq 3$ and all $%
n\geq 1$,
\begin{equation}
\mu _{n}=\frac{l_{2}(l_{3}^{k}-1)}{l_{3}-1}+l_{1}l_{3}^{k}  \label{mun}
\end{equation}%
where $n=2^{k}(2j+1)$ for some $k,j\geq 0$.
\end{lemma}

\begin{proof}
Let $\sigma:=\sigma(l_1,l_2,l_3)$. It is easy to see that $%
|\sigma^m(1)|_1=2^m $ and $|\sigma^m(0)|=|\sigma^m(0)|_0=l_3^m$ for any $%
m\geq 0$.

Now, we prove this result by induction on $n$. It is clear that $\mu_1=l_1$.
Assume that the result is true for all $n<2^m$; we prove it for $2^m\leq
n<2^{m+1} $. Since $\sigma^{\infty}(1)$ begins with $\sigma^{m+1}(1)$ and $%
|\sigma^{m+1}(1)|_1=2^{m+1}$, in order to evaluate $\mu_{n}$ for $2^m\leq
n<2^{m+1}$, we only need to investigate $\sigma^{m+1}(1)$. Note that $%
\sigma^{m}(1)$ begins with `$1$' and
\begin{eqnarray}
\sigma^{m+1}(1) &=& \sigma^m(\sigma(1))=\sigma^m(1\overbrace{0\cdots 0}%
^{l_1}1\overbrace{0\cdots 0}^{l_2})  \nonumber \\
&=&\sigma^m(1)\overbrace{\sigma^m(0)\cdots \sigma^m(0)}^{l_1}\sigma^m(1)%
\overbrace{\sigma^m(0)\cdots \sigma^m(0)}^{l_2}  \nonumber \\
&=&\sigma^{m-1}(1\overbrace{0\cdots 0}^{l_1}1\overbrace{0\cdots 0}^{l_2})%
\overbrace{\sigma^m(0)\cdots \sigma^m(0)}^{l_1}\sigma^m(1)\overbrace{%
\sigma^m(0)\cdots \sigma^m(0)}^{l_2}  \nonumber \\
&=&\sigma^{m-1}(1)\overbrace{\sigma^{m-1}(0)\cdots\sigma^{m-1}(0)}%
^{l_1}\sigma^{m-1}(1) \overbrace{\sigma^{m-1}(0)\cdots\sigma^{m-1}(0)}^{l_2}
\nonumber \\
& &\overbrace{\sigma^m(0)\cdots \sigma^m(0)}^{l_1}\sigma^m(1)\overbrace{%
\sigma^m(0)\cdots \sigma^m(0)}^{l_2};  \label{iteration_m}
\end{eqnarray}
then we have
\begin{equation}
\mu_{2^m}=\sum_{i=0}^{m-1}l_2l_3^i+l_1l_3^m=\frac{l_2(l_3^m-1)}{l_3-1}%
+l_1l_3^m.  \label{lem1eqn1}
\end{equation}
From (\ref{iteration_m}) and the fact $|\sigma^{m+1}(1)|_1=2^{m+1}$%
, we know that the block of consecutive zeros between $n$-th and $(n+1)$-th
`1' will appear in the second $\sigma^m(1)$ of $\sigma^{m+1}(1)$ and for all
$0<i<2^m$,
\begin{equation}
\mu_{2^m+i}=\mu_i.  \label{lem1eqn3}
\end{equation}
Since for any $i=2^k(2j+1)$ where $0<k<m$, we have $2^m+i=2^k(2j^{\prime}+1)$%
. Equation (\ref{lem1eqn3}) implies that for all $0<i<2^m$,
\begin{equation}
\mu_{2^m+i}=\mu_i=\frac{l_2(l_3^k-1)}{l_3-1}+l_1l_3^k.  \label{lem1eqn2}
\end{equation}
Therefore the result follows from  (\ref{lem1eqn1}) and  (%
\ref{lem1eqn2}).
\end{proof}



\begin{theorem}
\label{thm2} Let $\mathbf{c}$ be the Cantor-like
sequence satisfying  
$7l_{3}\geq 4(l_{1}+l_{2})+17$ and $l_{1}(l_{3}-1)+l_{2}>3$
, and $S$ be the sum-free set corresponding to $\mathbf{c}$. Then the sequence $(S_{n})_{n\geq 0}$ is $2$%
-regular.
\end{theorem}

\begin{proof}
According to Theorem \ref{thm:regular}, we only need to show that the gap
sequence $(\mu _{n})_{n\geq 1}$ of $\mathbf{c}$ is $2$-regular and
satisfies (\ref{hm}). By Lemma \ref{lemma1}, we have
\begin{equation}
\left\{
\begin{array}{l}
\mu _{2n}=l_{3}\mu _{n}+l_{2}, \\
\mu _{2n+1}=l_{1},%
\end{array}%
\right.\label{mu:rec}
\end{equation}%
which implies $(\mu _{n})_{n\geq 1}$ is $2$-regular. 

Now we will show that $(\mu _{n})_{n\geq 1}$ satisfies (\ref{hm}). By (\ref{mu:rec}), for all $m\geq 1$, 
\begin{eqnarray}
\sum_{i=1}^{2^m-1}\mu_i
&=& \sum_{i=0}^{2^{m-1}-1}\mu_{2i+1}+\sum_{i=1}^{2^{m-1}-1}\mu_{2i}\nonumber\\
&=&2^{m-1}l_{1}+l_{3}\sum_{i=1}^{2^{m-1}-1}\mu_{i}+(2^{m-1}-1)l_{2}.\nonumber
\end{eqnarray}
Thus
\begin{equation}
l_{3}\sum_{i=1}^{2^{m-1}-1}\mu_{i}=\sum_{i=1}^{2^m-1}\mu_i-2^{m-1}(l_{1}+l_{2})+l_{2}.\label{sum:mu}
\end{equation}

By (\ref{lem1eqn3}), we know that $(\mu _{n})_{n\geq 1}$ satisfies the second equation of (\ref{hm}). We will show the first equation of (\ref{hm}) by induction. Since $\mu_{2}=l_{3}l_{1}+l_{2}>l_{1}+3=\mu_{1}+2+\frac{3-1}{2}$, the first equation of (\ref{hm}) holds. Now suppose $\mu _{2^{m}}>\sum_{i=1}^{2^{m}-1}\mu _{i}+2^{m}+\frac{3^{m}-1}{2}$. Then by (\ref{mu:rec}) and (\ref{sum:mu}), we have 
\begin{eqnarray}
\mu_{2^{m+1}} &=& l_{3}\mu_{2^{m}}+l_{2}\nonumber\\
&\geq& l_{3}(\sum_{i=1}^{2^{m}-1}\mu _{i}+2^{m}+\frac{3^{m}-1}{2}+1)+l_{2}\nonumber\\
&=& \sum_{i=1}^{2^{m+1}-1}\mu_{i}-2^{m}(l_{1}+l_{2})+l_{3}(2^{m}+\frac{3^{m}+1}{2})+2l_{2}.\label{eqn:thm3:1}
\end{eqnarray}
Since $7l_{3}\geq 4(l_{1}+l_{2})+17$, we have $l_{3}-l_{1}-l_{2}\geq 2-\frac{3(l_{3}-3)}{4}$. Then
\begin{eqnarray}
& & 2^{m}(l_{3}-l_{1}-l_{2})+l_{3}\frac{3^{m}+1}{2}\nonumber\\
&>& 2^{m}\left(2-\frac{3(l_{3}-3)}{4}\right)+l_{3}\frac{3^{m}}{2}\nonumber\\
&\geq& 2^{m+1}+3^{m}\left(\frac{l_{3}}{2}-\frac{3}{4}\left(\frac{2}{3}\right)^{m}(l_{3}-3)\right)\nonumber\\
&\geq& 2^{m+1}+3^{m}\left(\frac{l_{3}}{2}-\frac{1}{2}(l_{3}-3)\right)=2^{m+1}+\frac{3^{m+1}}{2}.\label{eqn:thm3:2}
\end{eqnarray}
By (\ref{eqn:thm3:1}) and (\ref{eqn:thm3:2}), we know that $(\mu _{n})_{n\geq 1}$ satisfies (\ref{hm}).
\end{proof}

In the following, we give two examples to illustrate Theorem \ref{thm2}.
Example \ref{eg:const} gives a sum-free set
corresponding to a zero-one sequence generated by a substitution of constant length.
Example \ref{eg:nonconst} gives a sum-free set corresponding to a sequence generated
by a substitution of non-constant length.
\begin{example}\label{eg:const}
Take $\sigma:=\sigma(3,0,5)$, that is, $\sigma:~1\rightarrow 10001,~0\rightarrow
00000$. The sum-free set corresponding to
$\sigma^{\infty}(1)=100010000000000000001000100000\cdots$
is
$$S=\{1,6,24,29,110,115,133,138,528,533,551,556,637,642,660,665,\cdots\}.$$
According to Theorem \ref{thm2}, $S$ is $2$-regular.
\end{example}

\begin{example}\label{eg:nonconst}
Take $\sigma:=\sigma(1,1,5)$, that is $\sigma:~1\rightarrow 1010,~0\rightarrow
00000$. The sum-free set corresponding to
$$\sigma^{\infty}(1)=10100000010100000000000000000000000000000001010101000000\cdots$$
is
$$S=\{1, 3, 15, 17, 69, 71, 83, 85, 333, 337, 349, 353, 415, 417, 431, 435,\cdots\}.$$
According to Theorem \ref{thm2},
$S$ is $2$-regular.
\end{example}

From Theorem \ref{thm2}, we have the following two results.
In the following, the symbol $\equiv $ means equality modulo $2$.
\begin{cor}\label{cor:thue}
Let $(S_{n})_{n\geq 0}$ be the sum-free set in Theorem \ref{thm2}. For $%
0\leq j\leq 3$, the subsequence $(S_{4n+j})_{n\geq 0}$ is either periodic or
the Thue-Morse sequence up to a coding.
\end{cor}

\begin{proof}
Recall that by (\ref{sum:alpha}), we have 
$$\sum_{i=1}^{2^n-1}\alpha_i=\frac{3^{n}-1}{2}.$$ Using (\ref{sum:mu}) several times, we have 
\begin{eqnarray*}
\sum_{i=1}^{2^{n}-1}\mu_{i}&=&l_{3}\sum_{i=1}^{2^{n-1}-1}\mu_{i}+[2^{n-1}(l_{1}+l_{2})-l_{2}]\\
&=&l_{3}^{2}\sum_{i=1}^{2^{n-2}-1}\mu_{i}+l_{3}[2^{n-2}(l_{1}+l_{2})-l_{2}]+[2^{n-1}(l_{1}+l_{2})-l_{2}]\\
&=&\cdots \\
&=&\sum_{i=1}^{n}l_{3}^{i-1}[2^{n-i}(l_{1}+l_{2})-l_{2}]\\
&=&\frac{l_{1}+l_{2}}{l_{3}-2}(l_{3}^{n}-2^{n})-l_{2}\frac{l_{3}^{n}-1}{l_{3}-1}.
\end{eqnarray*}
By (\ref{alphan}), (\ref{mun}) and the previous two equations, we have 
\begin{eqnarray*}
\sum_{i=1}^{2^{n-1}}(\mu _{i}+\alpha _{i})&=&\sum_{i=1}^{2^{n-1}-1}(\mu _{i}+\alpha _{i})+\mu_{2^{n-1}}+\alpha_{2^{n-1}}\\
&=&\frac{l_{1}+l_{2}}{l_{3}-2}(l_{3}^{n-1}-2^{n-1})-l_{2}\frac{l_{3}^{n-1}-1}{l_{3}-1}+\frac{3^{n-1}-1}{2}\\
&&+(\mu_{2^{n-1}}+\alpha_{2^{n-1}})\\
&=&\frac{(l_{1}+l_{2})(l_{3}^{n-1}-2^{n-1})}{l_{3}-2}%
+l_{1}l_{3}^{n-1}+3^{n-1}.
\end{eqnarray*} 
Now by (\ref{hn}), we obtain
\begin{eqnarray}
h(n) &=&\left( \sum_{i=1}^{2^{n-1}}(\mu _{i}+\alpha _{i})\right) +2^{n-1}\nonumber\\
&=&\frac{(l_{1}+l_{2})(l_{3}^{n-1}-2^{n-1})}{l_{3}-2}%
+l_{1}l_{3}^{n-1}+3^{n-1}+2^{n-1}.\label{eqn:hn}
\end{eqnarray}%
Since $\frac{l_{3}^{n-1}-2^{n-1}}{l_{3}-2}=\sum_{i=0}^{n-2}l_{3}^{i}2^{n-2-i}\equiv l_{3}$ ($n\geq 3$), then
\begin{eqnarray*}
h(n) &\equiv &\frac{(l_{1}+l_{2})(l_{3}^{n-1}-2^{n-1})}{l_{3}-2}+l_{1}l_{3}+1
\\
&\equiv &(l_{1}+l_{2})l_{3}+l_{1}l_{3}+1 \\
&\equiv &1+l_{2}l_{3}.
\end{eqnarray*}

Let $(j)_{2}:=j_{2}j_{1}$ where $0\leq j\leq 3$. Then
\begin{eqnarray}
S_{4n+j} &=&1+\sum_{i=1}^{m}\epsilon _{i}h(i+2)+j_{2}h(2)+j_{1}h(1)\nonumber\\
&\equiv &1+(1+l_{2}l_{3})t_{n}+j_{2}h(2)+j_{1}h(1),\label{eqn:thue}
\end{eqnarray}%
where $t_n$ is the $n$-th term of the Thue-Morse sequence.
By (\ref{eqn:thue}), when $1+l_{2}l_{3}\equiv 0$, $(S_{4n+j})_{n\geq 0}$ modulo $2$ is a constant
sequence. When $1+l_{2}l_{3}\equiv 1$, $(S_{4n+j})_{n\geq 0}$ modulo $2$ is
the Thue-Morse sequence up
to a coding.
\end{proof}

\begin{example}
  According to Corollary \ref{cor:thue}, the sum-free set $S$ (modulo $2$) in Example
  \ref{eg:const} is the Thue-Morse sequence beginning by $1$. For
  the sum-free set $S$ in Example \ref{eg:nonconst}, the subsequences
$(S_{2n})_{n\geq0}$ and $(S_{2n+1})_{n\geq0}$ modulo $2$ are both the Thue-Morse sequence beginning
by $1$.
\end{example}

\begin{cor}
Let $S$ be the sum-free set corresponding to the sequences of Cantor type (i.e.,
the fixed point of $\sigma (l,0,l+2)\ (l\geq 2)$ beginning by $1$).

\begin{enumerate}
\item When $l$ is odd, then $(S_{n})_{n\geq 0}$ modulo $2$ is the Thue-Morse
sequence $(1-t_{n})_{n\geq 0}$.

\item When $l$ is even, then $(S_{2n})_{n\geq 0}\equiv(S_{2n+1})_{n\geq 0}$
modulo $2$, and they are both the Thue-Morse sequence $(1-t_{n})_{n\geq 0}$.
\end{enumerate}
\end{cor}

\begin{proof}
In this case, $l_{1}=l,l_{2}=0$ and $l_{3}=l+2$. Thus by  (\ref{eqn:hn}), $%
h(1)=l+2\equiv l$ and
\[
h(n)=l^{2}+3l+3^{n-1}\equiv 1,\ \ \ \forall~ n>1.
\]

When $l\equiv 1$, by  (\ref{eqn:thue}),
\begin{eqnarray*}
S_{4n+j} &\equiv &1+t_{n}+j_{2}+j_{1} \\
&\equiv &1+(\sum_{i=1}^{m}\epsilon _{i}+j_{2}+j_{1}) \\
&\equiv &1+t_{4n+j}.
\end{eqnarray*}%
Thus for any $n\geq 0$, $S_{n}\equiv 1+t_{n}.$

When $l\equiv 0$, for $n\geq 0$ and $j=0,1$,
\begin{eqnarray*}
S_{2n+j} &=&1+\sum_{i=1}^{m}\epsilon _{i}h(i+1)+j h(1) \\
&\equiv &1+\sum_{i=1}^{m}\epsilon _{i}\equiv 1+t_{n}.
\end{eqnarray*}%
Thus for any $n\geq 0$, $S_{2n}\equiv S_{2n+1}\equiv 1+t_{n}.$
\end{proof}

While $\mu_n$ increases fast, the corresponding sum-free set is not complicated. In fact, we have
\begin{proposition}
Suppose the sequence $(\mu_n)_{n\geq1}$ is increasing and
\begin{equation}  \label{munn}
\mu_{n+1}>2\sum_{i=1}^n\mu_i.
\end{equation}
Let $S_0=1$. Then for any $n\geq1$,
\[
S_n=\sum_{i=1}^n\mu_i+\frac{(n+1)(n+2)}{2}.
\]
\end{proposition}

\begin{proof}
By Lemma \ref{lm4}, for every $n\geq1$, we have $\alpha_{n+1}=\alpha_n+1$
and
\[
\alpha_n=\alpha_{n-1}+1=\alpha_{n-2}+1+1=\cdots=\alpha_1+(n-1)=n,
\]
since $\alpha_1=1$. Hence
\[
S_n=S_0+\sum_{i=1}^n\mu_i+\sum_{i=1}^n\alpha_i+n=\sum_{i=1}^n\mu_i+\frac{%
(n+1)(n+2)}{2}.
\]
\end{proof}

\begin{remark}
$(1)$ The growth order of $\mu _{n}$ is larger than $O(3^{n}).$ In fact, if $%
\mu _{n}=2\sum_{i=1}^{n-1}\mu _{i}$ for any $n\geq 3$, then $\mu _{n}=3\mu
_{n-1}$, and $\mu_2=2\mu_1$. Hence $\mu _{n}=2\times 3^{n-2}\mu _{1}.$

$(2)$ The coefficient $2$ in $(\ref{munn})$ is not crucial.
\end{remark}

\begin{remark}
Assume $(S_{n})_{n\geq 0}$ is a sequence given by $S_{n}=1+\sum_{i=1}^{m}%
\epsilon _{i}h(i)$ where $(n)_{2}:=\epsilon _{m}\cdots \epsilon_1 $ and $%
(h(i))_{i\geq 1}$ is a positive integer sequence. If $h(i)\equiv 1$ for $%
i\geq 0$, then $(S_{n})_{n\geq 0}$ modulo $2$ is the Thue-Morse sequence.
If $h(i)\equiv 0 $ for $i\geq 0$, then $S_{n}\equiv 1$.
\end{remark}

\begin{example}
Let $h(i)=2^{i}$, then $S$ is sum-free. Moreover, $S=\{1,3,5,7,9,\cdots \}$
and it is periodic of period $2$.
\end{example}

\begin{example}
Let $h(i)=(i+1)!$, then $S$ is sum-free. Moreover,
\[
S=\{1,3,7,9,25,27,31,33,\cdots \}
\]%
and it is periodic of period $4$.
\end{example}

\bigskip

\section{Base changing}

The mapping $\theta :\Sigma \rightarrow \mathfrak{S}$ is bijective, so it is
natural to study the properties of the corresponding zero-one sequences
of some sum-free sets. In this section, we will show that the corresponding
zero-one sequences of the sum-free sets defined below are automatic.

\begin{definition}
\label{def1} For any $b\geq 2,~n\geq 0$, let $S=(S_{n})_{n\geq 0}$ be the
sum-free set given by $S_{n}=[(n)_{b}1]_{2b-1}$.
\end{definition}

\begin{definition}[Allouche and Shallit \protect\cite{jps}]
Let $\mathcal{A}$ be a set of non-negative integers. Then we say that $%
\mathcal{A}$ is a $k$\emph{-automatic set} if its characteristic sequence
\[
a_{n}=\left\{
\begin{array}{ll}
1, & \hbox{if $n\in \mathcal{A}$;} \\
0, & \hbox{otherwise,}%
\end{array}%
\right.
\]%
defines a $k$-automatic sequence.
\end{definition}

\begin{lemma}
\label{lem2} The set $S$ in Definition \ref{def1} is sum-free and $(2b-1)$%
-automatic.
\end{lemma}

\begin{proof}
By the definition of $S_n$, it is clear that for any integer $n\geq0$, $%
(S_n)_{2b-1}\in\Sigma_b^{\ast}1$. Hence, for any integers $m,~n\geq0$, $%
(S_m+S_n)_{2b-1}\in\Sigma_{2b-1}^{\ast}2$, which implies that $%
S_m+S_n\not\in S$. Thus, the set $S$ is sum-free.

Let $(a_{n})_{n\geq 0}$ be the characteristic sequence of $S$. Then
\[
a_{n}=\left\{
\begin{array}{ll}
1, & \hbox{if $(n)_{2b-1}\in\Sigma_b^{\ast}1$;} \\
0, & \hbox{otherwise.}%
\end{array}%
\right.
\]%
Thus the sequence $(a_{n})_{n\geq 0}$ can be generated by the $(2b-1)$%
-automaton in Figure 1, which implies that the sequence $(a_{n})_{n\geq 0}$
is $(2b-1)$-automatic.
\begin{figure}[tbp]
\centering
\begin{tikzpicture}[shorten >=1pt,node distance=2cm,on grid,>=stealth,every state/.style={draw=blue!50,very thick,fill=blue!20}]
    \node[state,initial]  (q_0)                      {$q_0/0$};
    \node[state] (q_1) [below left=of q_0,yshift=-1.5cm,xshift=-2cm] {$q_1/1$};
    \node[state] (q_2) [below right=of q_0,yshift=-1.5cm,xshift=2cm] {$q_2/0$};
    \path[->]
    		(q_0) edge [bend right=15]  node [above left]  {1} (q_1)
                    		  edge   node [above right]  {$b,\cdots,2b-2$} (q_2)
    		(q_1) edge [bend right=15]   node [below right]  {$0,2,\cdots,b-1$} (q_0)
    			edge  node [below ]  {$b,\cdots,2b-2$} (q_2)
    		(q_0) edge [loop above] node {$0,2,\cdots,b-1$} ()
    		(q_1) edge [loop below] node {1} ()
                	(q_2) edge [loop below] node {$0,1,\cdots,2b-2$} ();
    \end{tikzpicture}
\label{AUT}
\caption{Automaton generating the set $(a_n)_{n\geq 0}$ in Lemma \protect\ref%
{lem2}.}
\end{figure}
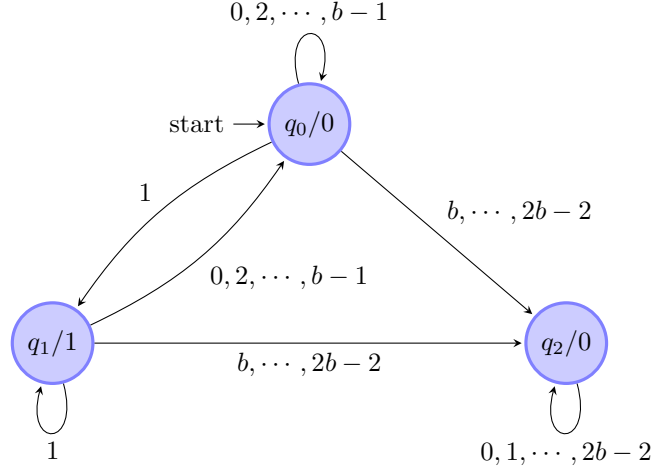
\end{proof}

\begin{theorem}
\label{the1} The zero-one sequence $\mathbf{c}$ corresponding to the
sum-free set $S$ in Definition \ref{def1} is $(2b-1)$-automatic.
\end{theorem}

\begin{proof}
Recall that the zero-one sequence $\mathbf{c}$ is obtained by deleting
all the $\ast $'s of the sequence $\mathbf{v}=(v_{n})_{n\geq 1}$. For any $%
n\geq 1$, we claim that

\begin{enumerate}
\item $v_n=1 \Leftrightarrow (n)_{2b-1}\in\Sigma_b^{\ast}1,$

\item $v_n=\ast \Leftrightarrow (n)_{2b-1}\in\Sigma_{2b-1}^{\ast}2,$

\item $v_n=0 \Leftrightarrow
(n)_{2b-1}\in\Sigma_{2b-1}^{\ast}\backslash(\Sigma_b^{\ast}1\cup%
\Sigma_{2b-1}^{\ast}2).$
\end{enumerate}

The third assertion is an immediate consequence of the first one and
the second one, and the first assertion follows directly from the
definitions of $S$ and $v_{n}$ in (\ref{vn}). Therefore it suffices to prove
the second assertion.

Since $v_n=\ast \Leftrightarrow n\in S+S$, we need to show that
\begin{equation}  \label{star}
S+S=(2b-1)\mathbb{N}+2.
\end{equation}

For any $S_n\in S$, then $S_n=(2b-1)([(n)_b]_{2b-1})+1.$ Hence, for any
integer $m,n\geq1$, we have
\[
S_m+S_n=(2b-1)([(m)_b]_{2b-1}+[(n)_b]_{2b-1})+2\in(2b-1)\mathbb{N}+2.
\]

Conversely, for any $n\geq 0$, assume $(n)_{2b-1}:=a_{k}a_{k-1}\cdots a_{1}$%
. Hence, for any $1\leq i\leq k$, there exist $b_{i},d_{i}\in \Sigma _{b}$
such that $a_{i}=b_{i}+d_{i}$. Thus, there exist two integers
\begin{eqnarray*}
n_{1} &=&[b_{k}b_{k-1}\cdots b_{1}]_{2b-1}, \\
n_{2} &=&[d_{k}d_{k-1}\cdots d_{1}]_{2b-1},
\end{eqnarray*}%
such that $n=n_{1}+n_{2}$ and $(n_{1})_{2b-1},(n_{2})_{2b-1}\in \Sigma
_{b}^{\ast }.$ Moreover,
\begin{eqnarray*}
(2b-1)n+2 &=&((2b-1)n_{1}+1)+((2b-1)n_{2}+1) \\
&=&S_{[b_{k}b_{k-1}\cdots b_{1}]_{b}}+S_{[d_{k}d_{k-1}\cdots d_{1}]_{b}}\in
S+S.
\end{eqnarray*}%
This implies  (\ref{star}) holds.

Now, we will show that $\mathbf{c}$ is $(2b-1)$-automatic. By (\ref{star}%
), $v_{(2b-1)n+2}=\ast $ for $n\geq 0$. Hence, $\mathbf{c}%
=c_{0}c_{1}\cdots $ satisfies
\begin{equation}
\left\{
\begin{array}{lll}
c_{(2b-2)n} & = & v_{(2b-1)n+1}; \\
c_{(2b-2)n+i} & = & v_{(2b-1)n+i+2},%
\end{array}%
\right.  \label{77}
\end{equation}%
where $1\leq i\leq2b-3.$ By Lemma \ref{lem2}, $S$ is $(2b-1)$-automatic. Its
characteristic sequence $(a_{n})_{n\geq 0}$ is a $(2b-1)$-automatic sequence.
By Theorem 6.8.1 in \cite{jps}, $(a_{(2b-1)n+i})_{n\geq 0}$ is $(2b-1)$%
-automatic for $0\leq i\leq 2b-2$. Note that $v_{n}=a_{n}$ for any $n\notin $
$(2b-1)\mathbb{N}+2$, then $(v_{(2b-1)n+i})_{n\geq 0}$ is $(2b-1)$-automatic
for $0\leq i\leq 2b-2$ and $i\neq 2$. Thus $(c_{(2b-2)n+i})_{n\geq 0}$ is $%
(2b-1)$-automatic for $0\leq i\leq 2b-3$. By Theorem 6.8.2 in \cite{jps}, $%
\mathbf{c}$ is $(2b-1)$-automatic.
\end{proof}

\begin{remark}\label{rem:cantor}
$(1)$ $(2b-1)$ in Theorem \ref{the1} cannot be replaced by $p$, where $p>2b-1$%
, since there do not exist $x,y\in \mathbb{N}$ such that $S+S=x\mathbb{N}%
+y$. \newline
$(2)$ From Theorem \ref{the1}, if $b=2$, then the sequence $\mathbf{c}$ is the
Cantor sequence
\[
101000101000000000101000101\cdots
\]%
$(3)$ If we replace $S_{n}=[(n)_{b}1]_{2b-1}$ by $S_{n}=[(n)_{b}w]_{2b-1}$
where $w\in \Sigma _{b}^{\ast }$, then the corresponding zero-one sequence
is also automatic.
\end{remark}

\medskip

\noindent\textbf{Acknowledgements }
We would like to thank Professor Jacques Peyri\`{e}re for introducing us to this topic, and we gratefully acknowledge
of his many helpful suggestions. We would also like to thank the anonymous 
referee for many helpful comments.

\end{document}